\documentclass[final]{siamltex}

\usepackage{amsmath,amsfonts,epsfig,amssymb}

\newtheorem{remark}[theorem]{Remark}

\newcommand{\real}{{\mathbb R}}
\newcommand{\posint}{{\mathbb Z}_+}
\newcommand{\integ}{{\mathbb Z}}
\newcommand{\nat}{{\mathbb N}}
\newcommand{\sC}{{\mathcal C}}

\newcommand{\sM}{{\mathcal M}}
\newcommand{\sP}{{\mathcal P}}

\newcommand{\sF}{{\mathcal F}}

\title{Convergence of moments of tau leaping schemes for unbounded Markov processes 
on integer lattices} 
\author{Muruhan Rathinam\thanks{ Mathematics and Statistics, University of
Maryland Baltimore County, 1000 Hilltop Circle, Baltimore, MD 21250, {\tt muruhan@umbc.edu}, Ph 410-455-2423, Fax 410-455-1066. This research was
supported in part by grant NSF DMS-0610013.} }

\begin{document}

\maketitle
%\date{}

\begin{abstract}
Tau leap schemes were originally designed for the efficient time stepping
of discrete state and continuous in time Markov processes arising in stochastic chemical kinetics.
Previous convergence results on tau leaping schemes have been restricted to 
systems that remain in a bounded subdomain (which may depend on the initial condition) or satisfy global
Lipschitz conditions on propensities. This paper extends the convergence results to fairly general tau leap schemes 
applied to unbounded systems that possess certain moment growth bounds. 
Specifically, we prove a weak convergence result, which shows order $q$
convergence of all moments under certain form of moment growth bound
assumptions on the stochastic chemical system and the tau leap method, as well 
as polynomial bound assumption on the propensity functions. 
The results are stated for a general class of Markov processes with 
$\integ^N$ as their state space. 
\end{abstract}


\thispagestyle{plain}

\section{Introduction}
The well stirred model of a chemical system as a continuous time 
Markov process with state space $\posint^N$ has been known for several
decades \cite{Gillespie76, Gillespie77, VanKampen}. Exact simulation of 
sample paths of such processes is very simple and is commonly known
as the SSA (abbreviation for Stochastic Simulation Algorithm) or the Gillespie
algorithm \cite{Gillespie77}. Stochastic chemical models have become important 
in applications in intracellular mechanisms and these models often possess some species in small molecular
copy numbers as well as a range of time scales in addition to nonlinear propensity functions. 
Hence approximations of the whole system by ordinary differential equations (ODEs) or even stochastic differential equations (SDEs) 
driven by Brownian motion is often not valid. 
On the other hand the SSA is often prohibitively expensive. Tau leaping methods 
were proposed as efficient but approximate alternatives to the SSA simulations. 

While the exact simulation (SSA) accounts for reaction events one at a time,
the tau leap methods take a predetermined time step and then provide an approximation of 
the random state at the end of the time step using some criterion. 
Thus tau leap simulation of
sample paths are akin to time stepping methods for ordinary 
differential equations (ODEs) and stochastic differential equations (SDEs)
driven by Brownian motion. The first tau leap method was proposed by
Gillespie \cite{Gillespie01} and is now known as the explicit tau leap method. 
This is in spirit the same as the explicit Euler method for ODEs. 
The implicit tau leap method 
was introduced in \cite{rathinam-petzold+03_jcp}
and the trapezoidal tau leap method may be 
found in \cite{cao-petzold-rathinam+04jcp}. Several other tau leap methods have
been proposed in the literature since then, see \cite{yang-rathinam+JCP11} for instance and references therein.

\subsection{Previous error analyses of tau leap methods}
As tau leap methods are analogous to the time
stepping methods for SDEs (driven by Brownian motion) and ODEs, the 
question of convergence is a natural one, where convergence is studied for a 
fixed time interval $[0,T]$ with mesh size $\max{(t_{j+1}-t_j)} \to 0$ where $0=t_0 < t_1 < \dots < t_n=T$ 
is the mesh used by the time stepping method. However, unlike the case of 
ODEs and SDEs, exact simulation is possible in the case of discrete state 
(continuous in time) Markov processes because the state of the process changes
via discrete events happening in continuous time. 
This means that if the step size of the tau leap method is very small one 
may expect on average no more than one event to occur during a time step, 
and hence the tau leap will no longer be more efficient than the exact simulation method! This fact has lead to interesting 
discussions and analyses. 

It was first shown in \cite{rathinam+05MMS} that both the explicit 
and the implicit tau methods are first order convergent in all moments for 
systems that remain in a bounded region (which may depend on the initial 
condition) of the state space under the assumption of linear propensity 
functions. It was later proven in \cite{Li-07MMS} that under the same 
bounded domain assumption but for general (nonlinear) propensity functions 
that the explicit tau method is first order convergent in moments 
as well as order $1/2$ convergent in a strong sense. Weak error analysis of 
explicit tau leap method with a ``Poisson bridge'' interpolation was provided 
in \cite{Karlsson-Tempone-11}.   

It must be noted that in the literature on 
numerical methods for stochastic dynamical systems the terms {\em strong
  error} and {\em weak error} are used in a slightly different sense from 
that of functional analysis. Strong error refers to 
the error $\hat{X}(t)-X(t)$ between the numerical approximation $\hat{X}$ of 
the process $X$ usually measured in the $L_1(\Omega,\sF,\text{Prob}))$ 
or $L_2(\Omega,\sF,\text{Prob})$ sense where $(\Omega,\sF,\text{Prob})$ is 
the common probability space which carries both the process $X$ as well as its
approximation $\hat{X}$. In the context of continuous time Markov processes 
on $\integ^N$ it is not always easy to find a good coupling of $X$
and $\hat{X}$ (unless one derives the method $\hat{X}$ starting from a
stochastic equation for instance with the aid of the random time change
representation \cite{EK-book} or with the aid of {\em Poisson random
  measures} \cite{Li-07MMS}) and there may be different ways to couple $X$ and
$\hat{X}$ leading to potentially different strong errors. Often one is 
interested in the error between the distribution of $X(t)$ and $\hat{X}(t)$. In particular for 
a function $f:\integ^N \to \real$ one considers the error
$E(f(\hat{X}(t)))-E(f(X(t)))$. This form of error analysis
is termed {\em weak error} analysis. Usually $f$ is taken to be a bounded
function on $\integ^N$ following the standard notion of {\em weak convergence}
of probability measures \cite{EK-book}. When the process $X$ as well as the
numerical scheme $\hat{X}$ remain in a bounded subset of $\integ^N$ no
assumption on $f$ is needed. However, it must be noted that when the process $X$ is not bounded 
and $f$ is taken to be a polynomial of degree higher than $2$, 
strong $L_2$ convergence will require additional regularity conditions in order to imply the convergence of $E(f(\hat{X}(t)))$ to $E(f(X(t)))$.       

When the molecular copy numbers are large, the stochastic chemical model 
may be well approximated by the reaction rate ODEs
\cite{Gillespie01}. This behavior is known as the {\em thermodynamic limit} in
 the applied sciences literature where one considers starting with the 
initial number of molecular copy numbers and the corresponding system volume, 
and then envisages a sequence of systems obtained by multiplying the initial 
copy numbers as well as the system volume by an integer $N$ and considering 
the behavior as $N \to \infty$. In order to obtain a limit, one must rescale
the process by $N$ and additionally a specific form of dependence of the 
propensities (probabilistic form of reaction rates) on the system volume is
critical for this limiting behavior to occur. This specific form of volume 
dependence or more abstractly ``system size'' dependence occurs commonly 
in many real world systems including stochastic chemical kinetics and is referred to as {\em density
  dependence} in the works of T.G. Kurtz where a rigorous proof of the limit
is also provided, see
\cite{EK-book} for instance. 

A natural question is how does a tau leap method 
behave when the system size becomes large. Some tau leap methods resemble
higher order numerical schemes for ODEs while the other tau leap methods 
resemble lower order schemes. This has motivated researchers to
incorporate system size into the error analysis of tau leap methods. The first
such analysis appeared in \cite{Anderson-Ganguly+AAP11} where the analysis 
investigates the explicit tau method as well as the midpoint
tau method. In particular the error analysis is carried out 
under the setting where the step size $\tau$ is related to system size
$V$ in the form of $\tau = V^{-\beta}$. This analysis is able to explain why 
when system size is sufficiently large the midpoint tau method performs better
than the explicit tau method. This analysis is also able to
explain why tau leap methods are effective while still leaping over several 
reaction events, when system size is sufficiently large. A system size dependent weak error analysis also
appears in \cite{Hu-Li+CMS11} where a rooted directed graph representation is
developed for weak Taylor expansions.  A weak error analysis under more general form of scaling with system size for general tau leap methods is presented in 
\cite{Anderson-Koyama-MMS12}. A related result shows that a large class of 
split step implicit tau leap methods limit to the implicit Euler scheme 
in the large volume limit while step size $\tau$ is fixed
\cite{Yang-Rathinam-JCompP13}. 

All convergence results for tau leap methods mentioned above \cite{rathinam+05MMS,
  Li-07MMS, Karlsson-Tempone-11, Anderson-Ganguly+AAP11,Anderson-Koyama-MMS12,Hu-Li+CMS11}
effectively apply only to systems that remain in a bounded domain. In
particular the Lipschitz or bounded derivative assumptions on propensity functions are only valid for either systems with
linear propensity functions or systems that remain in a bounded domain. 
While closed chemical systems satisfy the boundedness assumption 
due to conservation of atoms, in practice the assumption of a closed system 
is restrictive. Several models of biochemical systems have production of 
chemical species captured by reactions that may be described abstractly 
in the form $S \rightarrow S + A$. 

Related but different error analyses of time stepping methods for 
stochastic processes with jumps may be found in 
\cite{Hausenblaus-02, Glasserman-Merener-04, Higham-Kloeden-05} to mention 
a few. These articles are concerned with stochastic equations driven by 
Brownian motion and Poisson random measures. The first two works
\cite{Hausenblaus-02, Glasserman-Merener-04} consider fairly general jump 
processes but assume coefficient functions to be globally Lipschitz or possess 
bounded derivatives. The work in \cite{Higham-Kloeden-05} proves convergence 
of moments under the less restrictive one-sided Lipschitz condition on the 
drift coefficient but nevertheless assumes global Lipschitz condition on the 
coefficients corresponding to the Brownian and Poisson processes.  
Moreover the Poisson process considered has fixed intensity. None of these 
results are applicable to the chemical kinetic models with nonlinear 
propensities when the system is unbounded.   

\subsection{Error analysis in this work}
The important feature of the weak convergence result proved in this paper is 
that it does not assume boundedness of the system and moreover in the error 
$E(f(\hat{X}(t)))-E(f(X(t)))$ the function $f$ need not be
bounded, but is assumed to satisfy a polynomial growth bound. A form of moment 
growth bound (as a function of time) is assumed on the process $X$ 
and one may find sufficient conditions in \cite{rathinam-QAM-14, Engblom-14, Gupta-Briat+PLOS14} that ensure such bounds. 
The result applies to any tau leap method provided  
that it yields integer valued states, satisfies similar moment growth bound 
conditions as the chemical system, possesses pointwise local error of order 
$q+1$ and in addition satisfies certain bounds on the time 
derivative of moments. The analysis technique does not differentiate 
between explicit or implicit methods and applies to both provided they 
satisfy above conditions. The convergence proof does not apply to the 
(unrounded) implicit tau for instance since it yields noninteger states. 
However, it applies to split step implicit methods such as those in \cite{yang-rathinam+JCP11}.  
 
The proof technique involves establishing 
consistency and uniform boundedness (or zero stability) of the method 
in a certain family of norms and related metrics in the space of probability 
measures on a finite 
dimensional integer lattice which possess finite moments of all orders.  
Thus the proof is more in the spirit of the proof technique for ODEs 
though the spaces are infinite dimensional.  It must be noted that the 
notion of {\em zero stability} (see \cite{Ascher-Petzold-book} for instance) 
of a numerical scheme is an important concept. Essentially any sensible
numerical scheme closely approximates the exact process over one time step 
$\tau$ which is sufficiently small. But as $\tau \to 0$, the number of steps 
over a finite interval $[0,T]$ increases to $\infty$, and zero-stability 
requires that the numerical scheme is well behaved (uniformly bounded) under
this situation. 

%\subsection{Discussion of the types of error analyses}
The analysis in this paper does not consider scaling with system size into
account as is done in \cite{Anderson-Ganguly+AAP11,Anderson-Koyama-MMS12}. 
For the analysis in this paper, the system size $V$ is fixed while step size $\tau$
approaches zero. There has been some debate about which type of analysis is
better or even ``correct''. In other words, whether the step size $\tau$ 
should be taken as a function of system size parameter $V$,
 typically in the form of $\tau = V^{-\beta}$, and study the limiting behavior
 as $V \to \infty$, or following the more conventional analysis (where $V$ is
 fixed), study the limiting behavior as $\tau \to 0$. 
While the system size analysis provides valuable insights, a serious criticism 
of taking step size $\tau$ as a function $\tau =V^{-\beta}$ of $V$ is that 
the quantity $V$ is a given and not under the control of the user, 
while the step size $\tau$ is. Thus halving the step size $\tau$ to
``check for convergence'' will not be captured by this type of analysis. 
A good discussion highlighting 
the benefits of both types of analysis may be found in \cite{Hu-Li+CMS11} 
and we agree with the sentiments expressed there in that both types of analysis are relevant.    
Regarding the importance of fixed $V$ analysis, it must be emphasized that 
if a tau leap method is not zero stable or not convergent then the user is 
potentially operating on a shaky ground. 
%To see this, in a large system, there may be
%several reaction channels firing at different propensities (rates) 
%and practitioners may find themselves in situations where the $\tau$ 
%step may leap over several reaction events on average for most reaction
%channels but for some reaction channels the step may only leap over zero, one or two events on average.  In such a situation the use
%of a scheme that is not zero stable
% (i.e.\ for fixed $V$ the method behaves badly as $\tau \to 0$ with number of
% steps approaching $\infty$) 
%will be problematic.   
To put this another way, when using a zero stable method a practitioner only 
needs to worry about whether $\tau$ is small enough when it comes to accuracy 
issues. On the other hand if the 
practitioner uses a method that is not zero stable (s)he has to worry about 
whether $\tau$ is large enough as well as small enough, a very unsettling 
situation! 
Thus we believe that this form of convergence (or at least zero stability) 
is necessary and that the analysis represents an important improvement 
over previous results in that it accommodates unbounded systems with nonlinear propensities.

%\subsection{Role of tau leaping}
%As mentioned earlier, the SSA is computationally expensive in several reaction systems
%arising in stochastic models of chemical kinetics. These models often possess a large range of molecular
%copy numbers as well as a large range of time scales in addition to nonlinear propensity functions. 
%As such deterministic or even SDE (driven by Brownian motion) approximations of the whole system is not valid. 
%Several approximate methods which involve treating different parts of the system differently have been proposed 
%(see for instance \cite{haseltine-rawlings-02, rao-arkin-03, cao-gillespie+05_ssSSA}). 
%A systematic study based on asymptotic stochastic analysis which unifies these 
%different approaches may be found in \cite{Ball-Kurtz+06}.  
  
\subsection{Outline of the paper}
The rest of the paper is organized as follows. Section 2 deals with 
mathematical preliminaries and proves some results which are relevant for 
the convergence proof. Section 3 presents the convergence proof. 
Section 4 provides some results on the verification of the assumptions that underly the
convergence proof. Second part of Section 4 specifically considers tau leap
methods using Poisson and binomial updates which are common to most tau leap
methods. Section 5 provides some concluding remarks.  

\section{Mathematical setup and preliminaries}
\subsection{Chemical process and tau leap approximation}
\label{sec-process-tau}
We shall be concerned with continuous time Markov chains that take 
values on the state space $\integ^N$ that have certain specific structure.
The origin of this structure comes from stochastic models of chemical kinetics
where $N$ different molecular species undergo $M$ different reaction 
channels, and hence our rationale for the term {\em 
chemical process}. The state of a stochastic chemical process 
is an $N$ dimensional (nonnegative) integer vector
such that the $i$th component of the vector stands for total the number of
molecules of the $i$th species.
The specific structure dictates that for any given state $x \in \posint^N$ 
there are at most $M$ other states that the process can jump to and
the possible jump sizes are independent of the state $x$ and time $t$.
These jump sizes are {\em stoichiometric vectors}
$\nu_1,\dots,\nu_M \in \integ^N$ which correspond to the $M$ different reaction 
channels. Associated with each stoichiometric vector $\nu_j$ there is
a jump rate or {\em propensity} (in the chemical kinetics terminology)
$a_j(x)$ which in general is a function of the state $x$.
We define $a_0(x)$ by $a_0(x) = \sum_{j=1}^M a_j(x)$.
In our general result in Section 3 we consider the slightly more general (than
the chemical kinetic systems) case
where the state space is $\integ^N$. In Section 4 we mostly specialize to the
case of non-negative state space. 

Given $N,M \in \nat$, {\em stoichiometric vectors} 
$\nu_1,\dots,\nu_M \in \integ^N$, and {\em propensity functions} 
$a_j:\integ^N \rightarrow \real$ for $j=1,\dots,M$, we define the associated 
{\em chemical process} $X(t)$ for $t \in [0,\infty)$ to be a 
$\integ^N$ valued Markov process which only admits jump sizes 
$\nu_1,\dots,\nu_M \in \integ^N$ with corresponding intensities
$a_j(x)$ for $j=1,\dots,M$. This means that given $X(t)=x$, the 
waiting time for the next jump event is exponentially distributed with 
rate $a_0(x)$ and the probability that the next jump is of size $\nu_j$ 
is $a_j(x)/a_0(x)$. We shall consider the version of $X(t)$ that has
right continuous paths with left hand limits (known as {\em cadlag}). We shall only be concerned with
chemical processes that are {\em non-explosive}, i.e.\ do not have infinitely 
many jumps in any finite time interval. 

Given a chemical process $X(t)$ with $N$ species and $M$ reaction channels, 
we may define the transition probabilities
$P:[0,\infty) \times \integ^N \times \integ^N \rightarrow \real$ by
\begin{equation}
P(\tau,x,x^\prime) = \text{Prob}\{ X(t+\tau)=x^\prime | \, X(t)=x \}.
\label{eq_def_P}
\end{equation}
By the non-explosivity assumption, it follows that for each $\tau \geq 0$, we have 
\[
\sum_{x^\prime \in \integ^N} P(\tau,x,x^\prime) = 1.
\]
For each $\tau \geq 0$, $P(\tau,x,x^\prime)$ 
is an infinite matrix indexed by $x,x^\prime \in \integ^N$.

Throughout this paper we shall be concerned with infinite matrices indexed by 
$\integ^N$, i.e.\  functions $\psi:\integ^N \times \integ^N \rightarrow \real$. Such a matrix $\psi$ may be naturally regarded also as 
a linear operator $\psi$ from a subspace of $\real^{(\integ^N)}$ into 
$\real^{(\integ^N)}$ by the 
prescription that given $g \in \real^{(\integ^N)}$
we define $\psi \, g \in \real^{(\integ^N)}$ by the matrix vector multiplication (in reverse order)
\[
(\psi \, g)(y) = \sum_{x \in \integ^N} \psi(x,y) g(x),
\]
provided the sum converges absolutely.  
Given two operators (matrices) $\psi_1,\psi_2$ the ``product'' notation $\psi_1 \psi_2$ 
shall mean the composition $\psi_1 \circ \psi_2$ of operators which is also given by
the matrix multiplication in {\em reverse order}
\[
(\psi_1 \psi_2)(x,x') = \sum_{y \in \integ^N} \psi_1(y,x') \psi_2(x,y),
\]   
again when the sum above converges absolutely. 
Given such an operator $\psi$ we denote by $|\psi|$ the function $(x,x') \mapsto |\psi(x,x')|$ and like wise
given a function $g \in \real^{(\integ^N)}$  we denote by $|g|$ the function $x \mapsto |g(x)|$. 

Since there are only finitely many jumps out of each state, the time evolution of $P(\tau)$
satisfies the Kolmogorov's forward equation  
\begin{equation}
P^{(1)}(\tau,x,x^\prime) = \sum_{j=1}^M \left(P(\tau,x,x^\prime-\nu_j) a_j(x^\prime-\nu_j) - P(\tau,x,x^\prime) a_j(x^\prime) \right).
\label{eq_forward}
\end{equation}
Let us define $Q: \integ^N \times \integ^N  \rightarrow \real$ by
\begin{equation}
\begin{aligned}
Q(x,x') &= a_j(x), \quad \text{if } x'=x+\nu_j,\\
Q(x,x') &= -a_0(x), \quad \text{if } x'=x,\\
Q(x,x') &= 0, \quad \text{otherwise}.
\end{aligned}
\label{eq_Q}
\end{equation}
Thus we may write \eqref{eq_forward} as
\[
P^{(1)}(\tau,x,x') = \sum_{y \in \integ^N} Q(y,x') P(\tau,x,y),
\]
and this may be compactly written in operator notation as
\begin{equation}
P^{(1)}(\tau) = Q\, P(\tau).
\label{eq_forward_op}
\end{equation}
When regarded as an operator on $l_1(\integ^N;\real)$, $Q$ is known as the {\em generator}
of the semigroup $P(\tau)$. It must be noted that $Q$ is an unbounded operator and 
its domain is not all of $l_1(\integ^N;\real)$. The above operator equation holds
on the domain of both sides. 
Since the sum on the righthand side of \eqref{eq_forward} involves finitely many terms,
we may differentiate it arbitrary number of times. In operator notation we 
obtain that for $q \in \posint$,
\begin{equation}
P^{(q)}(\tau) = Q^q \, P(\tau).
\label{eq_forward_op_q}
\end{equation}
We note that $Q^q$ is well defined as a function on $\integ^N \times \integ^N$ 
or an infinite matrix since any given row or column of $Q$ has only 
finitely many nonzero entries and hence $q$-fold multiplication of $Q$ is well defined.  

Given a chemical process $X$ let $R(t) \in \posint^M$ denote the vector 
of reaction counts during $(0,t]$; in other words, for $j=1,\dots,M$, $R_j(t)$ is the number of times reaction channel $j$ fires during $(0,t]$.  
 If $X(t)=x$ then $X(t+\tau) = x + \nu (R(t+\tau)-R(t))$. 
For given $x$ and $\tau$, the conditional distribution (conditioned on
$X(t)=x$) of the random 
variable $R(t+\tau)-R(t)$ (which depends only on $x$ and $\tau$) is in general not known and hence it is difficult to generate a sample from.  
A tau leap method typically provides an approximation of the conditional distribution 
of $R(t+\tau)-R(t)$ given $X(t)=x$ by an easily computable random variable $K$
whose distribution depends on $x$ and $\tau$ and thus also provides an
approximation for the distribution of $X(t+\tau)$ by that of $x + \nu K$. 

In a very general sense, 
given (current) state $x \in \integ^N$ and a time step $\tau>0$ 
a {\em tau leap method} assigns an (approximate) probability mass function for 
the state $x'$ after elapsed time $\tau$. 
Thus we take the view point that a tau leap method is uniquely characterized by a map  
$\phi:[0,\infty) \times \integ^N \times \integ^N \rightarrow \real$ 
where $\phi(\tau,x,x^\prime)$ is the probability assigned to state $x^\prime$. 

We shall define a {\em mesh} $\Pi$ on $[0,T]$ to be a finite length sequence 
$\Pi=(t_0,\dots,t_n)$ that satisfies $0=t_0<t_1<\dots<t_{n-1}<t_n=T$. 
We shall define step sizes associated with $\Pi$ to be $\tau_j=t_j-t_{j-1}$
for $j=1,\dots,n$ and we shall denote the maximum step size 
$\max\{\tau_1,\dots,\tau_n\}$ by $|\Pi|$. 
Given a tau leap method $\phi$ and a mesh $\Pi=(t_0,\dots,t_n)$ on $[0,T]$ 
the {\em tau leap solution} $Y_\Pi(t)$ for $t \in [0,T]$ 
corresponding to initial condition $x_0 \in \integ^N$ is defined to be
the stochastic process which is constant on $[t_{j-1},t_j)$ for $j=1,\dots,n$
(thus jumps at $t_1,\dots,t_n$), satisfies $Y_\Pi(0)=x_0$ and also satisfies
\begin{equation}
\phi(\tau_j,x,x^\prime) = \text{Prob}\{Y_{\Pi}(t_j)=x^\prime | \, Y_\Pi(t_{j-1})=x \}, \quad j=1,\dots,n.
\label{eq_def_tau_process}
\end{equation}
Note that the tau leap solution $Y_{\Pi}(t)$ on any given mesh $\Pi$ is also 
a Markov process, but it is not time homogeneous since the family
$\phi(\tau)$ does not possess the semigroup property
with respect to the time parameter $\tau$.

We note that elements of $l_1(\integ^N;\real)$ may be regarded as signed finite measures on $\integ^N$ and denote by $\sP$ the set of all probability measures on $\integ^N$. 
We finally note that for each $\tau \geq 0$, the operators (or infinite matrices) 
$P(\tau)$ and $\phi(\tau)$ (which we call the {\em transition functions} of the process and the tau leap method respectively) have induced norm equal to $1$ (on $l_1(\integ^N;\real)$) and moreover
they leave $\sP$ invariant, i.e.\ map probabilities to probabilities. 

\subsection{Total variation, moment variation, spaces $\sM$ and $\sC$}
\label{sec_momvar}
In this section we define some spaces that shall play an important role 
in our convergence study. We remark up front that the spaces defined here are
weighted $l_1$ spaces and their duals. Related but different 
spaces (weighted $l_2$ and related discrete Sobolev spaces) were developed in \cite{Engblom-09}
for the spectral approximation of the solution of equation
\eqref{eq_forward}. 

First we recall the {\em total variation} norm. 
Given two signed finite measures $g_1$ and $g_2$ on $\integ^N$ 
the total variation between $g_1$ and $g_2$ is 
given by the $1$-norm distance 
\[
\|g_1-g_2\|_1 = \sum_{x \in \integ^N} |g_1(x)-g_2(x)|.
\]   

Throughout this paper we shall use $|.|$ to denote a norm on $\real^N$. 
For each $r \in \posint$ we shall define the $r$th {\em moment variation}
  $|.|_r$ 
on $l_1(\integ^N;\real)$ by
\begin{equation}
| g |_r = \sum_{x \in \integ^N} \frac{1}{2} \, (1+|x|^r) |g(x)| \leq \infty,
\label{eq_seminorm_r}
\end{equation}
for all $g \in l_1(\integ^N;\real)$.
We define the subspaces 
$\sM_r \subset l_1(\integ^N;\real)$ for $r \in \posint$ by
\begin{equation}
\sM_r = \{ g \in l_1(\integ^N;\real)\, | \;\; | g|_r < \infty \}
\label{eq_space_Mr}
\end{equation}
and $\sM$ by $\sM = \bigcap_{r \in \posint} \sM_r$.  It follows that $| .|_r$ 
is a norm on $\sM_r$ for each $r \in \nat$ and when $r=0$, $|.|_0$ is the total variation norm or equivalently the $1$-norm 
($\sM_0 = l_1(\integ^N;\real)$). We note that $\sM_r$ equipped with $|.|_r$
norm is a Banach space isometrically isomorphic to $l_1$, the space of summable
sequences. To see this let $\xi:\nat \to \integ^N$ be a bijection. 
Define $\eta:\sM_r \to l_1$ by 
\[
{\eta(g)}(n) = g(\xi(n)) (1+|\xi(n)|^r)/2.
\]
It is straightforward to verify that $\eta$ is an isometric isomorphism. 

It must also be noted that $\sM_r$ 
includes all probability measures which have a finite $r$th moment and
$\sM$ includes all probability measures that have finite moments of
all orders. 

\begin{remark}
Due to the equivalence of norms on $\real^N$, two different norms 
$|.|_r$ arising from two different norms on $\real^N$ are equivalent. 
%and hence so are the two metrics $d_r$ arising from these. 
\label{rem_equiv_norm}
\end{remark}

We state the following lemma which will be used frequently throughout this paper.

\begin{lemma}
For $0<r_1<r_2$ there exists $\alpha$ such that
\[
|g|_{r_1} \leq \alpha |g|_{r_2},
\]
for all $g \in \sM$.
\label{lem_r1lessr2}
\end{lemma}
\begin{proof}
The set of $x \in \integ^N$ for which $|x|<1$ is finite.
Thus there exists $\alpha$ such that $|x|^{r_1} \leq \alpha |x|^{r_2}$ for all $x \in \integ^N$.
\end{proof}

\begin{corollary} For $r \in \posint$, 
$\sM_{r+1} \subset \sM_r$.
\end{corollary}

%\subsection{Lattice functions of polynomial growth: classes $\sC_r$ and $\sC$}
%\label{sec_classC}
The main convergence results in this paper are obtained under the assumption 
that the propensity functions are at most of polynomial growth. 
We define classes $\sC_r$ and $\sC$ to make this concept precise and prove 
some important results concerning the generator $Q$ 
under the polynomial growth assumption on propensities. In particular we 
show that under polynomial growth assumption on propensities, 
$Q$ maps $\sM$ into $\sM$. 

For each $r \in \posint$ the class $\sC_r$ of 
functions $f:\integ^N \rightarrow \real$ that are said to be of {\em polynomial growth} of degree $r$ 
are defined by the condition that $f \in \sC_r$ if and only if there exists $\alpha>0$ such that 
\[
|f(x)| \leq \alpha (|x|^r + 1), \quad \forall x \in \integ^N.
\]
We define the class $\sC$ by $\sC = \cup_{r \in \posint} \sC_r$.

It is easy to see that for each $r \in \posint$, $\sC_r$  
is a Banach space when equipped with the norm that is given by 
\[
\|f\| = \sup\{ 2 f(x)/(1+|x|^r) \, | \, x \in \integ^N \},
\] 
for $f \in \sC_r$. Moreover, $\sC_r$ can be naturally identified with the 
dual $\sM_r^*$ of $\sM_r$ with the pairing given by 
\[
\langle f, g \rangle = \sum_{x \in \integ^N} f(x) g(x),
\]
where $f \in \sC_r$ and $g \in \sM_r$. 

\begin{lemma}
Suppose $f: \integ^N \rightarrow \real$ is given by a polynomial in $|x|$ 
of degree $r$. 
Then $f \in \sC_r$.
\label{lem_classC_poly}
\end{lemma}
\begin{proof}
We note that if $0 \leq r_1 < r_2$  
then there exists $\alpha >0$ such that $|x|^{r_1} \leq \alpha (|x|^{r_2} +1)$
for all $x \in \integ^N$.
This follows because the 
set of $x \in \integ^N$ such that $|x| < 1$ is finite regardless of the
norm used. 
\end{proof}
 
The following corollary is immediate.
\begin{corollary}
A (multivariate) polynomial $f:\integ^N \rightarrow \real$ belongs to $\sC$.
\label{cor_classC_poly}
\end{corollary}
Also note that the definitions of $\sC_r$ and $\sC$ are independent 
of the norm used in $\real^N$. 

%The following lemma is very useful and
%the proof is omitted. 
%\begin{lemma}
%The class $\sC$ is closed under addition, multiplication and 
%domain shifts. To be precise, for all $f,g \in \sC$ and $y \in \integ^N$
%it follows that $f+g \in \sC$, $f \, g \in \sC$ and $h \in \sC$, where
%$h$ is defined by $h(x)=f(x-y)$ for all $x \in \integ^N$. 
%\label{lem_closure_C}
%\end{lemma}

The following lemma plays an important role in our convergence analysis.
\begin{lemma} 
Let $Q$ as defined in \eqref{eq_Q} correspond to a 
chemical system whose propensity functions are of class $\sC_{s}$ for some $s \in \posint$.
Then for each $r \in \posint$, 
there exists $B_r>0$ such that 
\[
| Q\, g |_r = \big{|} |Q\, g| \big{|}_r \leq \big{|} |Q| \, |g| \big{|}_r \leq B_r | g|_{s+r},
\]
for each  $g \in \sM$. 
Hence $Q\, \sM \subset \sM$ and $|Q| \, \sM \subset \sM$.  In particular the domain of the generator $Q$ contains $\sM$. (See Section \ref{sec-process-tau} for definition of absolute value $|Q|$ ). 
\label{lem_QM_M}
\end{lemma}

\begin{proof}
\[
\begin{aligned}
\big{|} |Q| \,  |g| \big{|}_r &= \frac{1}{2} \, \sum_{x' \in \integ^N}  (1+|x'|^r) \Big|\sum_{x \in \integ^N}|Q(x,x')| |g(x)| \Big|    \\
&= \frac{1}{2} \, \sum_{x \in \integ^N}  \sum_{j=1}^M (1+ |x+\nu_j|^r) a_j(x) |g(x)| + \frac{1}{2} \, \sum_{ x \in \integ^N}  (1+|x'|^r) a_0(x) |g(x)|, \\
\end{aligned}
\]
where we have used \eqref{eq_Q}. 
Since $a_j$ are of class $\sC_s$, there exists $\alpha$ 
independent of $x$ such that
\[
a_j(x) \leq a_0(x) \leq \alpha (|x|^s + 1),
\]
for all $x$. Additionally we have 
\[
1 + |x + \nu_j|^r \leq 1 + (|x|+|\nu_j|)^r \leq 2^r (|x|^r + |\nu_j|^r) + 1 \leq \beta (|x|^r + 1),
\]
for some $\beta$ independent of $x$. 
Thus we obtain that for some constants $\tilde{B}_r$ and $B_r$ the following holds for all $g$:
\[
\begin{aligned}
\big{|} |Q| \,  |g| \big{|}_r &\leq  \frac{1}{2} \,  \sum_{x \in \integ^N} \tilde{B}_r (|x|^s + 1) (|x|^r + 1) |g(x)|,\\
&\leq B_r \, |g|_{r+s}
\end{aligned}
\]
Note that we have used Lemma \ref{lem_classC_poly}.
\end{proof}

Finally we provide a lemma which shows that convergence in the moment
variation norm $|.|$ is equivalent to convergence of $E(f(X_n))$ to $E(f(X))$ 
for all $f \in \sC_r$. 

\begin{lemma}\label{lem-momvar-mom-conv}
 For $n \in \nat$, let $p_n, p \in \sM_r$ be probability measures. 
The following are equivalent:
\begin{enumerate}
\item $\lim_{n \to \infty} |p_n-p|_r =0$.
\item For every function $f:\integ^N \to \real$ that is of class $\sC_r$ 
we have
\[
\sum_{x \in \integ^N} f(x) p_n(x) \to \sum_{x \in \integ^N} f(x) p(x).
\]
\end{enumerate}
\end{lemma}
\begin{proof}
We note that the first statement is that of strong convergence of $p_n$ to $p$
in $\sM_r$ (equipped with $|.|_r$) and the second is that of weak convergence 
of $p_n$ to $p$ in $\sM_r$. Since $l_1$ possesses the {\em Schur property} 
which states that ``a weakly convergent sequence is also strongly convergent''\cite{conway-book}, 
and $\sM_r$ is isometrically isomorphic to $l_1$, the result follows.
\end{proof}

\section{Convergence analysis}
Given the same initial condition $p_0 \in \sM \cap \sP$ (an initial probability 
measure on $\integ^N$ with finite moments of all orders) and 
a mesh $\Pi=(t_0,\dots,t_n)$ on $[0,T]$, let the $p(t)$ and $\hat{p}_{\Pi}(t)$
describe the probability mass functions of the chemical process $X(t)$
and its tau leap approximation $Y_\Pi(t)$ both of which satisfy
$p(0)=\hat{p}(0)=p_0$. 
We shall prove the convergence of $\hat{p}_{\Pi}(t)$ to $p(t)$ for $t=t_i$ 
in the $r$th moment variation norm
%(the $d_r$ metric) 
under suitable assumptions.  
In this section $P(\tau)$ stands for the transition function of the 
chemical process, $\phi(\tau)$ stands for the 
transition function of the tau leap method as defined in Section \ref{sec-process-tau}.
In what follows we shall use $\hat{p}(t)$ suppressing the subscript $\Pi$ 
for brevity.

We state a few assumptions about the chemical process $X(t)$ 
and its tau leap approximation that may be needed in the convergence results presented in this section. We note that Section 4 addresses the question of verification of these assumptions. Assumption 1 holds in all  
stochastic chemical models we have encountered in the literature 
and results in \cite{rathinam-QAM-14, Gupta-Briat+PLOS14, Engblom-14} provide conditions under which Assumption 2 holds and 
Theorem \ref{thm_ass2} of Section 4 restates a special case of a result proved in 
\cite{rathinam-QAM-14} regarding Assumption 2. Theorems \ref{thm_Ass5}, \ref{thm_Ass3} and \ref{thm_ass6} of Section 4 provide some general 
conditions under which Assumptions 5, 3 and 6 hold respectively and Theorem \ref{thm2_ass5}, Corollary \ref{cor_ass3} and Theorem \ref{thm2_ass6} provide more 
specific conditions for tau leap methods where reaction counts are approximated by (conditioned on current state) independent Poisson and/or binomial random variables. 

{\bf Assumption 1: Polynomial growth bound on propensities} All propensity functions of the chemical process are
in class $\sC_{s^*}$ for some $s^* \geq 0$.

{\bf Assumption 2: Exponential moment growth bound for $P$.} 
For all $r \in \posint$ there exist $\lambda_r>0$ 
such that for all $\tau>0$ and all $x \in \integ^N$ the following holds:
\begin{equation}
\sum_{x^\prime \in \integ^N} (1+|x^\prime|^r) \, P(\tau,x,x^\prime) \leq (1 + |x|^r)\, e^{\lambda_r \tau}. 
\label{eq_mom_bound_P}
\end{equation}
We may state \eqref{eq_mom_bound_P} equivalently as
\begin{equation}
|P(\tau) g|_r \leq |g|_r \, e^{\lambda_r \tau}, \quad \forall g \in \sM 
\label{eq_mom_bound_P_norm}
\end{equation}

Yet another equivalent way to state Assumption 2 is
\begin{equation}
E(1 + |X(t+\tau)|^r \, | X(t) = x) \leq (1 + |x|^r )\, e^{\lambda_r \tau}.
\label{eq_equiv_Ass2}
\end{equation}

{\bf Assumption 3: Pointwise consistency of order $q$.} 
For each $x \in \integ^N$ 
and 
$x^\prime \in \integ^N$,  
$\phi(\tau,x,x')$ is $q+1$ times
continuously differentiable in $\tau$ and the following hold:
\begin{equation}
\phi^{(i)}(0,x,x') = P^{(i)}(0,x,x'), \quad i=1,\dots,q.
\label{eq_ptwse_cons}
\end{equation}

Note that it follows from the finite sum on the right hand side of 
the Kolmogorov's forward equations \eqref{eq_forward} that 
$P$ is infinitely differentiable in $\tau$, 
so we do not need the differentiability assumption for $P$.

{\bf Assumption 4: Derivative bound on $\phi$ in total variation norm.} 
There exist $H_0>0$, $s_0 >0$, $\delta_0>0$ and $\gamma_0>0$ such that 
for all $\tau \in [0,\delta_0)$
\begin{equation}
\sum_{x' \in \integ^N} |\phi^{(q+1)}(\tau,x,x')| \leq \frac{H_0}{2} ( |x|^{s_0} + 1) \, e^{\gamma_0 \tau},
\label{eq_deriv_bd_phi}
\end{equation}
where $q$ is as in Assumption 3. 
Equation \eqref{eq_deriv_bd_phi} may be equivalently stated as
\begin{equation}
\big{|} |\phi^{(q+1)}(\tau)| g\big{|}_0 \leq H_0 \, |g|_{s_0} \, e^{\gamma_0 \tau}, \quad \forall g \in \sM.
\label{eq_deriv_bd_phi_norm}
\end{equation}

{\bf Assumption 5: Derivative bound on $\phi$ in moment variation norms.} 
For each $r \in \posint$ there exist $H_r>0$, $s_r >0$, $\delta_r>0$ and $\gamma_r>0$ such that 
for all $\tau \in [0,\delta_r)$
\begin{equation}
\sum_{x' \in \integ^N} (1 + |x'|^r) \, |\phi^{(q+1)}(\tau,x,x')| \leq  H_r ( |x|^{s_r} + 1) \, e^{\gamma_r \tau},
\label{eq_mom_deriv_bd_phi}
\end{equation}
where $q$ is as in Assumption 3. 
Equation \eqref{eq_mom_deriv_bd_phi} may be equivalently stated as
\begin{equation}
\big{|}|\phi^{(q+1)}(\tau)| g\big{|}_r \leq H_r \, |g|_{s_r}  \, e^{\gamma_r \tau}, \quad \forall g \in \sM.
\label{eq_mom_deriv_bd_phi_norm}
\end{equation}

Note that Assumption 5 implies Assumption 4. 

{\bf Assumption 6: Exponential moment growth bound for $\phi$.} 
For each $r \in \posint$ there exist $\lambda_r>0$ and $\delta_r>0$ 
%and a function $F_r \in \sD$ 
such that for all $\tau \in [0,\delta_r)$ and all $x \in \integ^N$ the following holds:
\begin{equation}
\sum_{x^\prime \in \integ^N} (1+|x^\prime|^r) \, \phi(\tau,x,x^\prime) \leq (1 + |x|^r)\, e^{\lambda_r \tau}. 
\label{eq_mom_bound_phi}
\end{equation}
We may state \eqref{eq_mom_bound_phi} equivalently as
\begin{equation}
|\phi(\tau) g|_r \leq |g|_r \, e^{\lambda_r \tau}, \quad \forall g \in \sM 
\label{eq_mom_bound_phi_norm}
\end{equation}
Note that for convenience we have chosen without loss of generality $\lambda_r$ %and $F_r$ to 
to be the same as in \eqref{eq_mom_bound_P} of Assumption 2.

\begin{remark}
We note that when these assumptions are used, it is assumed that there exists 
a common norm on $\real^N$ such 
that Assumptions 1 through 6 hold (in that same norm). 
Assumptions 1 and 3 are independent of the norm used on $\real^N$. 
Under suitable sufficient conditions Assumption 2 may be shown to hold in 
any norm on $\real^N$ with constants $\lambda_r$ depending on the norm  
\cite{rathinam-QAM-14}.  It is straight forward to show that 
Assumptions 4 and 5 are independent of the norm as long as norm dependent constants $H_r$ 
are allowed.    
\label{rem_Ass_norm_indep}
\end{remark}

\begin{remark}
If deterministic initial condition is assumed then convergence results can be
obtained under slightly relaxed versions of the above assumptions. For
instance in Assumption 2 the constant $\lambda_r$ will be required to be
independent of $x$ only within the set of states reachable from the initial
condition and not independent of all $x \in \posint^N$. We shall not pursue
this line of inquiry for sake of brevity.  
\label{rem_Ass_det_IC}
\end{remark}

An equation similar to \eqref{eq_mom_deriv_bd_phi} 
follows for $P(\tau)$ under Assumptions 1 and 2,  
which we state as a lemma.
\begin{lemma}
\label{lem_mom_deriv_bd_P}
For each $r \in \posint$ there exist $H_r>0$, $s_r >0$ and $\gamma_r>0$ such that 
for all $\tau>0$,
\begin{equation}
\sum_{x' \in \integ^N} (1 + |x'|^r) \, |P^{(q+1)}(\tau,x,x')| \leq H_r ( |x|^{s_r} + 1) \, e^{\gamma_r \tau}
\label{eq_mom_deriv_bd_P}
\end{equation}
which may be equivalently stated as
\begin{equation}
\big{|}|P^{(q+1)}(\tau)| g\big{|}_r \leq H_r \, |g|_{s_r}  \, e^{\gamma_r \tau}, \quad \forall g \in \sM.
\label{eq_mom_deriv_bd_P_norm}
\end{equation}
\end{lemma}
Note that without loss of generality we may take $\gamma_r$, $s_r$, and $H_r$ to be the 
same in equations \eqref{eq_deriv_bd_phi}, \eqref{eq_mom_deriv_bd_phi} and \eqref{eq_mom_deriv_bd_P}.
\begin{proof}
\[
\begin{aligned}
\Big{|} |P^{(q+1)}(\tau)| \, g\Big{|}_r &\leq \Big{|} |P^{(q+1)}(\tau)| \, |g| \Big{|}_r = \Big{|} |Q^{q+1} \, P(\tau)|  \, |g| \Big{|}_r \leq \Big{|} |Q|^{q+1} \, P(\tau) \, |g| \Big{|}_r  \\ 
&\leq B_r B_{r+s^*} B_{r+2 s^*} \cdots B_{r+q s^*}\, \big{|}P(\tau) |g| \big{|}_{(q+1)s^*+r} \\
&\leq  B_r B_{r+s^*} \cdots B_{r+q s^*}\, |g|_{(q+1)s^*+r} e^{\lambda_{(q+1)s^*+r} \tau} \\
&\leq H_r \, |g|_{s_r} \, e^{\gamma_r \tau},
\end{aligned}
\]
where $H_r$, $s_r$ and $\gamma_r$ are suitably large, and we have used 
Lemma \ref{lem_QM_M} repeatedly and Assumptions 1 and 2. 
\end{proof}

The following consistency result follows from Assumptions 1 through 5 and 
Lemma \ref{lem_mom_deriv_bd_P}.

\begin{lemma}{\bf Order $q$ Consistency in moment variation.}
Suppose for a common norm on $\real^N$ the 
Assumptions 1 through 5 hold. (For $r=0$ case only Assumptions 1 through 4 are needed).
For each $r \in \posint$ let $s_r$, $\delta_r$ and $\gamma_r$ be as in 
\eqref{eq_mom_deriv_bd_phi} and \eqref{eq_mom_deriv_bd_P}.
Then for each $r \in \posint$ there exist $C_r>0$ 
such that for all  $\tau \in [0,\delta_r)$ and $g \in \sM$, 
\begin{equation}
| (\phi(\tau) - P(\tau)) g |_r \leq C_r \, | g|_{s_r+r}  \, \tau^{q+1} \,e^{\gamma_r \tau}
\label{eq_phi_minus_P_bd}
\end{equation}
\label{lem_consistency}
\end{lemma}
\begin{proof}
From \eqref{eq_mom_deriv_bd_phi} and \eqref{eq_mom_deriv_bd_P}  
we obtain using triangle
inequality that
\[
 \sum_{x' \in \integ^N}  (1+|x'|^r) \, |\phi^{(q+1)}(\tau,x,x') - P^{(q+1)}(\tau,x,x')|  \leq 2 H_r (1 + |x|^{s_r}) \, e^{\gamma_r \tau},
\]
for all $\tau>0$. 
From Taylor's theorem we have that for each $x, x' \in \integ^N$, and for each $\tau>0$,
\[
 \phi(\tau,x,x') - P(\tau,x,x') = \int_0^\tau \frac{1}{q!} \, \left(\phi^{(q+1)}(s,x,x') - P^{(q+1)}(s,x,x')\right) \, (\tau-s)^q ds.
\]
Hence 
\[
\begin{aligned}
&\sum_{x^\prime \in \integ^N} (1+|x'|)^r \, |\phi(\tau,x,x^\prime) - P(\tau,x,x^\prime)| 
\\
&= \frac{1}{q!} \sum_{x^\prime \in \integ^N} \left| \int_0^\tau (1+|x^\prime|^r) \,\left(\phi^{(q+1)}(s,x,x^\prime) - P^{(q+1)}(s,x,x^\prime)\right)\, (\tau-s)^q ds \right | \\
&\leq \frac{1}{q!} \sum_{x^\prime \in \integ^N} \int_0^\tau  (1+|x^\prime|^r) \, \left| \phi^{(q+1)}(s,x,x^\prime) - P^{(q+1)}(s,x,x^\prime)\right| \, (\tau-s)^q ds \\
&= \frac{1}{q!} \int_0^\tau \left( \sum_{x^\prime \in \integ^N} (1+|x^\prime|^r) \, \left| \phi^{(q+1)}(s,x,x^\prime) - P^{(q+1)}(s,x,x^\prime)\right| \right) \, (\tau-s)^q ds \\
&\leq \int_0^\tau 2 \frac{(\tau-s)^q}{q!} \, H_r \, ( |x|^{s_r} + 1) \, e^{\gamma_r s} ds \;
\leq 2 \frac{\tau^{q+1}}{q!} \, H_r \,( |x|^{s_r} + 1)\, e^{\gamma_r \tau} \\
&\leq \frac{2 H_r}{q!}  (|x|^{s_r} + 1) \, \tau^{q+1} \, e^{\gamma_r \tau},
\end{aligned}
\]
where  we have used the dominated convergence theorem to swap the sum and the integral.
Thus, given $g \in \sM$ we obtain
\[
\begin{aligned}
|(\phi(\tau)-P(\tau)) g|_r &\leq  \sum_{x \in \integ^N}\sum_{x^\prime \in \integ^N} \frac{1}{2} (1+|x|^r) \,|\phi(\tau,x,x^\prime) - P(\tau,x,x^\prime)| \, |g(x)| \\
 &\leq \sum_{x \in \integ^N} \frac{C_r}{2} (1 + |x|^{s_r+r}) \, |g(x)| \, \tau^{q+1} \, e^{\gamma_r \tau}\\
&= C_r \, |g|_{s_r+r} \, \tau^{q+1} \, e^{\gamma_r \tau}
\end{aligned}
\]
where $C_r$ is a suitably large constant.
\end{proof}

The following theorem establishes the order $q$ convergence in total
variation of a tau leap
method that is pointwise order $q$ consistent under the Assumptions $1$ through $4$.

\begin{theorem}{\bf Order $q$ convergence in total variation}\label{thm_totv_conv}
Let $\Pi=(t_0,\dots,t_n)$ be a mesh on $[0,T]$.
Let $p(t)$ and $\hat{p}_{\Pi}(t)$ for $t \in [0,T]$ be the probability 
mass functions corresponding to the stochastic chemical process 
and its tau leap approximation on mesh $\Pi$ both started with initial
distribution $p_0 \in \sM \cap \sP$. Let $\tau = |\Pi|$ be the maximum step size. 
Suppose for a common norm on $\real^N$ the Assumptions 1 through 4 hold and
$s_0$, $\delta_0$  and $\gamma_0$ be as in 
\eqref{eq_mom_deriv_bd_phi} and \eqref{eq_mom_deriv_bd_P} and
let $C_0$ be as in Lemma \ref{lem_consistency} for the case $r=0$ and let $\mu_0=\max\{\lambda_{s_0},\gamma_0\}$. 
Then for each $i=0,1,\dots,n$ and for $\tau \in (0,\delta_r)$
the following holds  :
\begin{equation}
|\hat{p}_{\Pi}(t_i) - p(t_i)|_0 \leq  C_0 \, |p_0|_{s_0}  \, t_i \, e^{\mu_0 t_i} \, \tau^q  \leq C_0 \, |p_0|_{s_0}  \, T \, e^{\mu_0 T} \, \tau^q.
\label{eq_conv_totvar}
\end{equation} 
\end{theorem}

\begin{proof}
For $i=1,\dots,n$ we may write the error 
$\hat{p}(t_i)-p(t_i)$ as
\[
\hat{p}(t_i)-p(t_i) = \phi(\tau_i) \left(\hat{p}(t_{i-1})-p(t_{i-1})\right) + \left(\phi(\tau_i) - P(\tau_i) \right) p(t_{i-1}).
\]
Repeated application of the above leads to the telescoping sum
\begin{equation}
\hat{p}(t_i)-p(t_i) = \sum_{j=1}^i \phi(\tau_i) \, \phi(\tau_{i-1}) \, \dots \, \phi(\tau_{j+1}) \, \left( \phi(\tau_j) - P(\tau_j)\right) \, p(t_{j-1})
\label{eq_err_telesc}
\end{equation}
where we have used the fact that $\hat{p}(0)=p(0)=p_0$. 
From \eqref{eq_phi_minus_P_bd}
\[
|(\phi(\tau_j) - P(\tau_j)) \, p(t_{j-1})|_0 \leq C_0 \, | p(t_{j-1})|_{s_0} \, \tau_j^{q+1} \,e^{\gamma_0 \tau_j}, 
\]
since $| p(t_{j-1})|_0=1$.
From \eqref{eq_mom_bound_P_norm} we obtain
\[
| p(t_{j-1})|_{s_0} = | P(t_{j-1}) p_0|_{s_0} \leq |p_0|_{s_0} \, e^{\lambda_{s_0} t_{j-1}}.
\]
Hence with $\mu_0 = \max\{\lambda_{s_0},\gamma_0\}$ we obtain
%\[
%|(\phi(\tau_j) - P(\tau_j)) \, p(t_{j-1})|_0  \leq C_0 \, |p_0|_{s_0} \, \tau^q% \, \tau_j e^{\mu_0 t_j}.
%\]
%Hence  
\begin{equation}
|(\phi(\tau_j) - P(\tau_j)) \, p(t_{j-1})|_0  \leq C_0 \, |p_0|_{s_0} \, \tau^q \, \tau_j e^{\mu_0 t_j}.
\label{eq_uniform_consist_var}
\end{equation}
The equation \eqref{eq_uniform_consist_var} is a statement of order $q$ {\em uniform consistency} in total variation norm on the interval $[0,T]$. 
Using the fact that $|\phi(\tau_i)|_0=1$ for all $i$, we obtain 
from \eqref{eq_err_telesc} the estimate
\[
|\hat{p}_{\Pi}(t_i) - p(t_i)|_0
\leq \sum_{j=1}^i  C_0 \, |p_0|_{s_0}  \, \tau^q \, \tau_j \, e^{\mu_0 t_j}
\leq C_0 \, |p_0|_{s_0}  \, t_i \, e^{\mu_0 t_i} \, \tau^q.
\]
This completes the proof.
\end{proof}

Now we have the following $0$-stability or uniform boundedness result
for the tau leap method which follows directly from Assumption 6. 

\begin{lemma}{\bf Uniform boundedness or zero stability of tau leap method in $r$th moment variation.}
For each $r \in \posint$, $T>0$ 
%there exist $\lambda_r>0$ (independent 
%of $T$) and $K_r(T)$ (possibly dependent on $T$) such that  for each 
$g \in \sM$, and for all meshes $\Pi=(t_0,\dots,t_n)$ on $[0,T]$ satisfying $|\Pi| < \delta_r$ and for 
any indices $i,j$ with $0 \leq j < i \leq n$ the following holds:
\begin{equation}
%\begin{aligned}
|\phi(\tau_i) \phi(\tau_{i-1}) \dots \phi(\tau_{j+1}) g|_{r} 
\leq |g|_{r} e^{\lambda_r (\tau_{j+1}+\dots+\tau_i)} \leq |g|_{r} e^{\lambda_r T}. 
%\end{aligned}
\label{eq_uniform_bdd_tau}
\end{equation}
\label{lem_uniform_bdd_tau}
\end{lemma}

The following theorem establishes the order $q$ convergence in $r$th
moment variation of a tau leap
method that is order $q$ consistent under the Assumptions $1$ through $6$.

\begin{theorem}{\bf Order $q$ convergence in moment variation}
Let $\Pi=(t_0,\dots,t_n)$ be a mesh on $[0,T]$.
Let $p(t)$ and $\hat{p}_{\Pi}(t)$ for $t \in [0,T]$ be the probability 
mass functions corresponding to the stochastic chemical process 
and its tau leap approximation on mesh $\Pi$ both started with initial
distribution $p_0 \in \sM \cap \sP$. Let $\tau = |\Pi|$ be the maximum step size. 
Suppose for some common norm on $\real^N$ the Assumptions $1$ through $6$ hold. Given any $r \in \posint$ let
$s_r$ and $\gamma_r$ be as in 
\eqref{eq_mom_deriv_bd_phi} and \eqref{eq_mom_deriv_bd_P}, 
let $C_r$ be as in Lemma \ref{lem_consistency} and let $\mu_r=\max\{\lambda_{s_r+r},\gamma_r\}$.

Then for each $r \geq 0$ and for each $i=0,1,\dots,n$ and $\tau \in (0,\delta_r)$
the following holds  :
\begin{equation}
|\hat{p}_{\Pi}(t_i) - p(t_i)|_{r} \leq C_r |p_0|_{s_r+r} \,t_i\, e^{\mu_r t_i} \,\tau^q  \leq C_r |p_0|_{s_r+r} \,T\, e^{\mu_r T} \,\tau^q. 
\label{eq_conv}
\end{equation} 
\label{thm_conv}
\end{theorem}

\begin{proof}
From \eqref{eq_phi_minus_P_bd}
\[
|(\phi(\tau_j) - P(\tau_j)) \, p(t_{j-1})|_{r} \leq C_r  | p(t_{j-1})|_{s_r+r}  \, \tau_j^{q+1} \,e^{\gamma_r \tau_j}. 
\]
From \eqref{eq_mom_bound_P_norm} we obtain
\[
| p(t_{j-1})|_{s_r+r} = | P(t_{j-1}) \, p_0|_{s_r+r} \leq | p_0|_{s_r+r} \, e^{\lambda_{s_r+r} t_{j-1}}.
\]
With $\mu_r = \max\{\lambda_{s_r+r},\gamma_r\}$ we obtain 
\begin{equation}
|(\phi(\tau_j) - P(\tau_j)) \, p(t_{j-1})|_{r}  \leq C_r  | p_0|_{s_r+r} \, \tau_j^{q+1}\, e^{\mu_{r} t_{j}}.
\label{eq_uniform_const}
\end{equation}
which is a statement of {\em uniform consistency}. 
In Lemma \ref{lem_uniform_bdd_tau}
for $i>j$ taking $g =  (\phi(\tau_j) - P(\tau_j)) \, p(t_{j-1})$ 
and using \eqref{eq_uniform_const}
we obtain the estimate
\[
\begin{aligned}
&|\phi(\tau_i) \phi(\tau_{i-1}) \dots \phi(\tau_{j+1}) (\phi(\tau_j) - P(\tau_j)) \, p(t_{j-1})|_{r}\\
&\leq C_r  | p_0|_{s_r+r} \, \tau_j^{q+1} e^{\mu_{r} t_{j}} e^{\lambda_r (\tau_{j+1}+\dots+\tau_i)} 
\leq C_r | p_0|_{s_r+r} \, \tau_j^{q+1} e^{\mu_{r} t_i}.
\end{aligned}
\]
Thus we obtain from \eqref{eq_err_telesc} the estimate
\[
\begin{aligned}
|\hat{p}_{\Pi}(t_i) - p(t_i)|_{r} &\leq \sum_{j=1}^i C_r |p_0|_{s_r+r} \, \tau^q \, \tau_j \, e^{\mu_r t_i}\\
&\leq C_r |p_0|_{s_r+r} \,t_i\, e^{\mu_r t_i} \,\tau^q  \leq C_r |p_0|_{s_r+r} \,T\, e^{\mu_r T} \,\tau^q 
\end{aligned}
\]
This completes the proof.
\end{proof}

The following corollary affirming the order $q$ convergence of moments is immediate.

\begin{corollary}{\bf Order $q$ convergence of moments}
Let the assumptions of Theorem \ref{thm_conv} hold. Then the error in the $r$th moment satisfies
\begin{equation}
|E(|Y_{\Pi}(T)|^r) - E(|X(T)|^r)| \leq 2\, C_r E(|X(0)|^{s_r+r}) \,T\, e^{\mu_r T} \,\tau^q.
\label{eq_conv_mom}
\end{equation} 
\end{corollary}
\begin{proof}
\[
\begin{aligned}
|E(|Y_{\Pi}(T)|^r) - E(|X(T)|^r)| &= \big{|} \sum_{x \in \integ^N} |x|^r \hat{p}_{\Pi}(T,x) - \sum_{x \in \integ^N} |x|^r p(T,x)\big{|} \\
\leq \sum_{x \in \integ^N} (1+|x|^r) \, |\hat{p}_{\Pi}(T,x) - p(T,x)|
&= 2 \, |\hat{p}_{\Pi}(T)-p(T)|_r \leq 2\, C_r |p_0|_{s_r+r} \,T\, e^{\mu_r T} \,\tau^q.
\end{aligned}
\]
\end{proof}

\begin{remark}\label{rem-Mr-M}
For convenience of exposition our convergence analysis and the 
Assumptions 2, 4 and 6 dealt with the situation where moments of all orders
exist. However it is clear from our analysis that our Assumptions 2, 4 and 6  along
with the assumption $p_0 \in \sP \cap \sM$ can be weakened to the case where 
moments exist only up to some order $r_0$. 
\end{remark}

\begin{remark}\label{rem-sup-error}
We note that using Assumption 2 it is straightforward to extend the
convergence results to obtain a first order supremum error bound of the form
\begin{equation}
\sup_{t \in [0,T]} |E(|Y_{\Pi}(t)|^r) - E(|X(t)|^r)| \leq \tilde{C}_r |p_0|_{s_r+r} T e^{\tilde{\mu}_r T} \tau,
\label{eq-sup-error}
\end{equation} 
where per our convention the tau leap approximation $Y_\Pi(t)$ is constant on $[t_{j-1},t_j)$. 
\end{remark}

\section{Verification of the conditions of the convergence theorem}

In this section we provide some results on the verification of 
Assumptions 1 through 6. All forms for propensity functions proposed in the literature 
that we have encountered satisfy the polynomial growth bound of Assumption $1$ and thus it is not 
restrictive. It is also straightforward to verify. 

\subsection{General results on verification of Assumption 2 through 6}
Firstly it must be noted that from \eqref{eq_forward_op_q} we have $P^{(i)}(0)
= Q^i$ for $i =1,2,\dots$ since $P(0)$ is the identity. This gives explicit
expressions for $P^{(i)}(0,x,x^\prime)$. 
%We show the case for $i=1$ to be 
%\[
%\begin{aligned}
%P^{(1)}(0,x,x^\prime) &= a_j(x), \quad x^\prime = x+ \nu_j,\\
%&= -a_0(x), \quad x^\prime = x,\\
%&= 0, \quad \text{otherwise},
%\end{aligned}
%\]
%provided $\nu_j$ are distinct and note that expressions for higher order
%derivatives (at $\tau=0$) are straightforward to derive.  
The pointwise consistency (Assumption 3) 
requires $\phi^{(i)}(0,x,x^\prime)$ to agree with $P^{(i)}(0,x,x^\prime)$ 
for $i=1,\dots,q$. So checking Assumption 3 relies on evaluating 
$\phi^{(i)}(0,x,x^\prime)$. If direct expressions are available for 
$\phi(\tau,x,x^\prime)$ this is easy to do. However, in practice the 
expressions for $\phi(\tau,x,x^\prime)$ may involve infinite sums. 
To see this, recall that one may write the change in the chemical process 
$X(t)$ as
\begin{equation}
X(t+\tau) = x + \sum_{j=1}^M \nu_j [R_j(t+\tau)-R_j(t)],
\label{eq_update}
\end{equation}
where $X(t)=x$ and $R_j(t)$ are processes that count the number of reactions that occurred during $(0,t]$.
Most tau leap methods are of the form 
\begin{equation}
Y(t+\tau) = x + \sum_{j=1}^M \nu_j K_j
\label{eq_tau_update}
\end{equation}
where $Y(t)=x$ and $K_j$ are random variables whose distribution 
depends on $x$ and $\tau$ and are approximations of $R_j(t+\tau)-R_j(t)$.
Let us define the conditional probabilities
\begin{equation}
\begin{aligned}
\tilde{\phi}(\tau,x;k) &= \text{Prob}(K=k \, | \, Y(t)=x),\\ 
\tilde{p}(\tau,x;k) &= \text{Prob}(R(t+\tau)-R(t)=k \, | \, X(t)=x).
\end{aligned}
\label{eq_tildephi}
\end{equation}
In order to see the relationship between $P$ and $\tilde{p}$ as well as $\phi$
and $\tilde{\phi}$, given a a pair of states $x,x^\prime \in \integ^N$,
we define the associated set $S(x,x^\prime) \subset \posint^M$ 
to be the set of all reaction counts $k \in \posint^M$ that would take the system 
from state $x$ to state $x'$:
\begin{equation}
S(x,x^\prime) = \{k \in \posint^M \, | \, x^\prime - x = \nu \, k  \}.
\label{eq_S}
\end{equation}
Then we have that for $x,x' \in \integ^N$,
\begin{equation}
\begin{aligned}
P(\tau,x,x') &= \sum_{k \in S(x,x')} \tilde{p}(\tau,x;k),\\
\phi(\tau,x,x') &= \sum_{k \in S(x,x')} \tilde{\phi}(\tau,x;k).
\end{aligned}
\label{eq_phi_tildephi}
\end{equation}
Since expressions for $\tilde{\phi}$ are more readily available than for
$\phi$, we shall seek pointwise consistency of $\tilde{\phi}$ with $\tilde{p}$. 
In order to go from pointwise consistency of $\tilde{\phi}$
with $\tilde{p}$ to that of $\phi$ with $P$, term by term
differentiation needs to be justified as $S(x,x')$ may be infinite. 

In order to derive pointwise consistency conditions for
$\tilde{\phi}(\tau,x;k)$ in comparison with $\tilde{p}(\tau,x;k)$ we first note that given $X(t)=x$, the
reaction count process $R(t+\tau)-R(t)$ is a Markov process and
hence we obtain the following Kolmogorov's forward equation:
\begin{equation}
\tilde{p}^{(1)}(\tau,x;k) = \sum_{j=1}^M \tilde{p}(\tau,x;k-e_j) a_j(x+\nu(k-e_j)) 
- \sum_{j=1}^M \tilde{p}(\tau,x;k) a_j(x + \nu k),
\label{eq_forward_R}
\end{equation}
with initial probability $\tilde{p}(0,x;0)=1$ and $\tilde{p}(0,x;k)=0$ for 
$k \neq 0$. Here $e_j$ is the vector with all zeros except a one on the $j$th
entry. 
Defining the infinite matrix $\tilde{Q}(x)$ that depends on state $x$ by
\begin{equation}
\begin{aligned}
\tilde{Q}(x;k',k) &= a_j(x+\nu k'),  \quad k = k' + \nu_j,\\
                  &= -a_0(x + \nu k'), \quad k = k',\\
                  &=0, \quad \text{else},
\end{aligned}
\end{equation} 
we note that 
\begin{equation}
\tilde{p}^{(i)}(0,x;k) = \tilde{Q}^i(x;0,k), \; \forall k \in \posint^M,
\label{eq-tildepi}
\end{equation}
where $\tilde{Q}^i$ is the $i$th power of $\tilde{Q}$. 
Thus pointwise consistency of order $q$ for $\tilde{\phi}(\tau,x;k)$ is given by
\begin{equation}
\tilde{\phi}^{(i)}(0,x;k) = \tilde{Q}^i(x;0,k), \; \forall k \in \posint^M, \; i=1,\dots,q.
\label{eq-tilde-consist}
\end{equation}
We note that for $q=1$, \eqref{eq-tilde-consist} yields that
$\tilde{\phi}^{(1)}(0,x;k) = a_j(x)$ if $k=e_j$, 
$\tilde{\phi}^{(1)}(0,x;0) = -a_0(x)$ and $\tilde{\phi}^{(1)}(0,x;k)=0$ for
all other $k$.  

The following theorem provides a set of 
sufficient conditions that guarantee the validity of the term by term
differentiation for the sums involving $\tilde{\phi}$ and also guarantee that the Assumption 5 (on the derivative
bounds) holds. 

\begin{theorem}
Suppose there exists $\delta>0$, such that  $\tilde{\phi}(\tau,x;k)$ are continuously differentiable (in $\tau$)  $q+1$ times for $\tau \in [0,\delta]$ 
and for each $x,k$, and suppose that for each $k$ and $i=0,1,\dots,q+1$ there exist $\mu_{k,i}(x)$ such that 
\[
|\tilde{\phi}^{(i)}(\tau,x;k)| \leq \mu_{k,i}(x),
\]
and that for each $r \in \posint$ there exist $\eta_{r,i}$ and $\sigma_{r,i}$ such that 
\[
\sum_{k \in \posint^M} |k|^r \mu_{k,i}(x) \leq \eta_{r,i}(1 + |x|^{\sigma_{r,i}}).
\]
Then Assumption 5 holds with $\delta_r=\delta$, $\gamma_r=0$, and some $s_r$ for all $r \in \posint$.
\label{thm_Ass5}
\end{theorem}    
\begin{proof}
First we note that using Weierstrass test, for $i=0,1,\dots,q+1$ and all $r \in \posint$, the series
\[
\sum_{k \in \posint^M} |k|^r \tilde{\phi}^{(i)}(\tau,x;k),
\]
converges uniformly for $\tau \in [0,\delta]$ and that the commutation
\[
\left(\sum_{k \in \posint^M} |k|^r \tilde{\phi}(\tau,x;k)\right)^{(i)} =  \sum_{k \in \posint^M} |k|^r \tilde{\phi}^{(i)}(\tau,x;k)
\]
holds. It is also then clear that \eqref{eq_phi_tildephi} may be differentiated term by term $q+1$ times:
\[
\phi^{(i)}(\tau,x,x') = \sum_{k \in S(x,x')} \tilde{\phi}^{(i)}(\tau,x;k).
\]
This leads to the estimate
\[
\begin{aligned}
&\sum_{x' \in \integ^N} |x'|^r |\phi^{(q+1)}(\tau,x,x')| \leq   \sum_{k \in \posint^M} |x + \nu k|^r |\tilde{\phi}^{(q+1)}(\tau,x;k)| \\
&\leq \sum_{l=0}^r \frac{r!}{l!(r-l)!} |x|^{r-l} \|\nu\|^l \left(\sum_{k \in \posint^M} |k|^l |\tilde{\phi}^{(q+1)}(\tau,x;k)|\right)\\
&\leq \tilde{\eta}_{r}(1 + |x|^{s_r}),
\end{aligned}
\]
where $\tilde{\eta}_{r}$ is a suitably large constant and $s_r$ is the maximum of $r-l+\sigma_{l,q+1}$ over $l=0,1,\dots,r$.
Assumption 5 follows with a suitably large $H_r$ and $\gamma_r=0$. 
\end{proof}

\begin{corollary}\label{cor-termbyterm}
Suppose the conditions of Theorem \ref{thm_Ass5} and Assumption 1 hold. Then
\eqref{eq_phi_tildephi} may be term by term differentiated $q+1$ times. 
\end{corollary}
\begin{proof}The result for $\phi$ follows from the proof of Theorem
  \ref{thm_Ass5}. Under Assumption 1, because of \eqref{eq_forward_R} it can
  be shown that $\tilde{p}$ satisfies conditions similar
  to those required on $\tilde{\phi}$ by Theorem \ref{thm_Ass5}. So the 
term by term differentiation for $P$ also follows.
\end{proof}   

The following theorem is immediate.
\begin{theorem}
Suppose the conditions of Theorem \ref{thm_Ass5} hold and additionally 
that Assumption 1 and \eqref{eq-tilde-consist} hold. 
Then Assumption 3 holds. 
\label{thm_Ass3}
\end{theorem}
%\begin{proof}
%By the proof of Theorem \ref{thm_Ass5} we know that \eqref{eq_phi_tildephi} can be differentiated term by term to show that $\phi^{%(1)}(0,x,x')=a_j(x)$ for 
%$x'=x+\nu_j$, $\phi^{(1)}(0,x,x') = -a_0(x)$ for $x'=x$ and for other choices of $x'$ we have $\phi^{(1)}(0,x,x')=0$ showing first %order pointwise consistency.
%\end{proof}

The Assumption $2$ involves the moment growth bound condition on the chemical process. 
Verifying these conditions may not be trivial. Some sufficient conditions for Assumption 2 may be 
found in \cite{rathinam-QAM-14, Engblom-14, Gupta-Briat+PLOS14}. We provide one result which follows from 
Theorem 3.6 of \cite{rathinam-QAM-14}. 

We shall say that a reaction channel $j$ is {\em linearly bounded} if there exists a
constant $H$ such that
\[
a_j(x) \leq H (1+|x|), \quad \forall x \in \posint^N.
\]
If a reaction channel is not linearly bounded we refer to it as {\em
  superlinear}. Let us denote by $M_s$ the number of superlinear reactions. In
what follows we assume without loss of generality that the reactions are
ordered such that the first $M_s$ are superlinear.

While our convergence analysis of Section 3 did not assume that the non-negative lattice $\posint^N$ was invariant for the process, 
the sufficient condition we provide here for Assumption 2  will only apply to systems that remain 
in $\posint^N$ when started in $\posint^N$. Such a process is said to be {\em
  conservative} with respect to $\posint^N$. Any realistic model of chemical
kinetics as well as other population processes must have this property. 
It is easy to see that the process $X$ is conservative with respect to 
$\posint^N$ if and only if for every $x \in \posint^N$ if $x + \nu_j \notin \posint^N$ then $a_j(x)=0$. 

\begin{theorem}
Suppose that $X$ is conservative with respect to $\posint^N$, Assumption 1 is
satisfied and that there exists $\alpha \in \posint^{N}$ such that $\alpha>0$ and 
 $\alpha^T \nu_j \leq 0$ for $j=1,\dots,M_s$. Assume $X(0) \in \posint^N$ with probability $1$. Then for each $r \in \nat$ there 
exists $\lambda_r$ such that the following holds for all $t \geq 0$ and in any norm $|.|$ on $\real^N$:
\[
E(|X(t)|^r) \leq E(|X(0)|^r) e^{\lambda_r t} + e^{\lambda_r t} - 1.
\]
\label{thm_ass2}
\end{theorem}
\begin{proof}
This is implied by the proof of Theorem 3.6 of \cite{rathinam-QAM-14}.   
\end{proof}

For $x \in \posint^N$, $l \in \posint$ and $\tau>0$ let us define $m_l(x,\tau)$ to be the $l$th moment of the vector copy number of the 
linearly bounded reactions over a time step $\tau$ starting with state $x$
according to the tau leap method:
\begin{equation}
m_l(x,\tau) = \sum_{k} |k^{(2)}|^l \tilde{\phi}(\tau,x;k).
\label{eq_ml}
\end{equation}
Here vector copy number of reaction counts $k$ is written as
$k=(k^{(1)},k^{(2)}) \in \posint^{M_s}\times \posint^{M-M_s}$ where $k^{(1)}$
is the vector copy number of superlinear reactions and $k^{(2)}$ is that of
linearly bounded ones. We note that $m_0=1$. 
  
The following theorem provides sufficient conditions that guarantee Assumption 6.

\begin{theorem}\label{thm_ass6}
Suppose that there exists $\alpha$ satisfying the hypotheses of Theorem
\ref{thm_ass2}. Suppose
further that for each $l \in \nat$ there exist $\beta_l >0, \tilde{\delta_l}>0$
such that for all $x \in \posint^{N}$ and $\tau \in [0,\tilde{\delta_l}]$,  
\begin{equation}
m_l \leq \beta_l (1 + |x|^l) \tau,
\label{eq_ml_bnd}
\end{equation}
and for $x \notin \posint^{N}$ suppose that $\tilde{\phi}(\tau,x;0)=1$
(i.e.\ $K=0$  with probability $1$) which means that if the tau leap scheme
leaves $\posint^N$ it is stopped. Furthermore suppose that if $x \in \posint^N$
and for $k=(k^{(1)},k^{(2)})$ if $x + \nu^{(1)} k^{(1)} \notin \posint^N$ then
  $\tilde{\phi}(\tau,x;k)=0$. (This means if $x \in \posint^N$ then the tau
  update of the superlinear reactions alone will still result in a state in
  $\posint^N$ with probability $1$).  
Then Assumption 6 holds in a particular norm. If in addition the conditions of Theorem \ref{thm_Ass5} hold then
Assumption 6 holds in any norm.  
\end{theorem}
\begin{proof}
Define the norm on $\real^N$ by $|x| = \sum_{i=1}^N \alpha_i |x|_i$. 
Then $|x+\nu_j| \leq |x|$ if $x \in \posint^N$ and 
$x + \nu_j \in \posint^N$ for $j=1,\dots,M_s$. We denote by $\nu^{(1)}$ the 
$N \times M_s$ sub-matrix consisting of superlinear reactions and by $\nu^{(2)}$ 
the $N \times (M-M_s)$ sub-matrix consisting of linearly bounded reactions. 
Then we have that for $\tilde{\phi}(\tau,x;k) \neq 0$ with $k=(k^{(1)},k^{(2)})$,
\[
|x+\nu k| = |x + \nu^{(1)}k^{(1)} + \nu^{(2)} k^{(2)}| \leq |x + \nu^{(1)}k^{(1)}| + |\nu^{(2)} k^{(2)}| \leq |x| + \|\nu^{(2)}\| |k^{(2)}|,
\]
where $\|\nu^{(2)}\|$ is the induced norm of $\nu^{(2)}$. Using this we get 
\[
\begin{aligned}
&\sum_{x'} |x'|^r \phi(\tau,x,x') = \sum_{k} |x+\nu k|^r \tilde{\phi}(\tau,x;k) \\
&\leq \sum_{k} \left(|x| + \|\nu^{(2)}\| |k^{(2)}|\right)^r \tilde{\phi}(\tau,x;k) \leq \sum_{l=0}^r \frac{r!}{l!(r-l)!} |x|^{r-l} \|\nu^{(2)}\|^l m_l(x,\tau).
\end{aligned}
\]
Using the bounds on $m_l$ we obtain that for suitably large $\lambda_r$ and suitably small $\delta_r>0$ we have
\[
\sum_{x'} (1+|x'|^r) \phi(\tau,x,x') \leq (1+|x|^r) (1 + \lambda_r \tau) \leq (1+|x|^r)e^{\lambda_r \tau},
\]
for all $\tau \in [0,\delta_r]$. This shows that Assumption 6 holds in the
particular norm defined above.

If in addition the conditions of Theorem \ref{thm_Ass5} hold then 
\[
\sum_{x'} (1+|x'|^r) \phi(\tau,x,x')
\]
is differentiable in $\tau$ and by Lemma 3.5 of \cite{rathinam-QAM-14} the 
Assumption 6 holds in any norm. 
\end{proof}

\begin{remark}\label{rem-ass6}
We note that proof of Theorem \ref{thm_ass6} uses an approach similar  
to that of Theorem \ref{thm_ass2} (see \cite{rathinam-QAM-14}) in that it 
is required that the reactions that have superlinear propensities 
are expected to decrease the norm of the state (in some norm). Since the 
original process remains non-negative the existence of 
$\alpha \in \posint^{N}$ such that $\alpha>0$ and  $\alpha^T \nu_j \leq 0$ 
for $j=1,\dots,M_s$ is adequate to ensure this. However in the case of 
a tau leap method we directly require that 
the superlinear reactions alone shall not result in a non-negative state in 
order to accomplish this. Thus it will be advisable to use bounded random 
variables such as Binomials for superlinear reactions to ensure
non-negativity. 
\end{remark}  

\subsection{Tau leap methods with Poisson and binomial updates}\label{sec-Poiss-bino}
Most tau leap methods use Poisson or binomial random variables for the $K_j$.
In this subsection we present further results that apply specifically to tau leap methods that use Poisson and binomial random variables. 

We first state some lemmas related to Poisson and binomial random variables. 
\begin{lemma}
\label{lem-poiss-poly-mom}
Let $K$ be Poisson distributed with parameter $\lambda$. Then for each $r \in
\posint$ the moment $E(K^r)$ is a polynomial in $\lambda$ of degree $r$. 
\end{lemma}
\begin{proof}
This follows via induction using the easy to establish recursion
\[
E(K^r) = \lambda E( (K+1)^{r-1}).
\]
\end{proof}

\begin{lemma}
\label{lem-bino-poly-mom}
Let $K$ be binomially distributed with parameters $N$ and $p$. Then for each
$r \in \posint$ the moment $E(K^r)$ is a polynomial of degree $r$ separately in
$N$ and $p$. 
\end{lemma}
\begin{proof}
This follows via induction using the easy to establish recursive relation
\[
E(K_N^r) = N p E((1+K_{N-1})^{r-1}),
\]
where $K_N$ denotes a binomial random variable with parameters $N$ and $p$. 
\end{proof}

\begin{lemma}
\label{lem-poiss-deriv-bnd}
Let $K$ be Poisson distributed with parameter $\lambda$ where $\lambda=\lambda(x,\tau)$
is a function of state $x \in \posint^N$ and step size $\tau \geq 0$. Denote $\psi(\lambda,k)$ the probability
that $K=k$. Suppose that there exists $\delta>0$ such that for all $x \in
\posint^N$ and $\tau \in [0,\delta]$, $\lambda$ is $q+1$ times continuously differentiable
in $\tau$, and the supremum of $\lambda,|\lambda^{(1)}|,\dots,|\lambda^{(q+1)}|$ over
$\tau \in [0,\delta]$ is bounded above by a polynomial in $|x|$. Then for each
$r \in \posint$ and $i=0,1,\dots,q+1$, the supremum of 
\[
\sum_k k^r |\psi^{(i)}(\lambda,k)|
\]
over $\tau \in [0,\delta]$ is bounded above by a polynomial in $|x|$. 
\end{lemma}
\begin{proof}
It is straight forward to verify the relation
\[
\psi^{(1)}(\lambda,k) = \lambda^{(1)} \left(\psi(\lambda,k-1)-\psi(\lambda,k) \right), \;\; k \in \posint,
\] 
where the convention that $\psi(\lambda,k) = 0$ for $k <0$ is used.  
By repeated application one can relate $\psi^{(i)}$ for $i=2,\dots,q+1$ also to $\psi$. This provides an upper bound for the quantities of interest in terms of the moments. Then the 
result follows by Lemma \ref{lem-poiss-poly-mom}.
\end{proof}

\begin{lemma}
\label{lem-bino-deriv-bnd}
Let $K$ be binomially distributed with parameters $N_0$ and $p$ where $N_0=N_0(x)$
is a function of state $x \in \posint^N$ and $p=p(x,\tau)$ is a function of
state $x$ and step size $\tau \geq 0$. Denote $\psi(N_0,p,k)$ the probability
that $K=k$. Suppose that there exists $\delta>0$ such that for all $x \in
\posint^N$ and $\tau \in [0,\delta]$, $p$ is $q+1$ times continuously differentiable
in $\tau$, and $N_0(x)$ as well as the supremum of $p,|p^{(1)}|,\dots,|p^{(q+1)}|$ over
$\tau \in [0,\delta]$ are bounded above by a polynomial in $|x|$. Then for each
$r \in \posint$ and $i=0,1,\dots,q+1$, the supremum of 
\[
\sum_k k^r |\psi^{(i)}(N_0,p,k)|
\]
over $\tau \in [0,\delta]$ is bounded above by a polynomial in $|x|$. 
\end{lemma}
\begin{proof}
It is straight forward to verify the relation
\[
\psi^{(1)}(N_0,p,k) = N_0 p^{(1)} \left( \psi(N_0-1,p,k-1) - \psi(N_0-1,p,k)
\right), \;\; k \in \{0,\dots,N_0\},
\] 
where the convention that $\psi(N_0,p,k)=0$ for $k \notin \{0,1,\dots,N_0\}$
is used. By repeated application one can relate $\psi^{(i)}$ for
$i=2,\dots,q+1$ also to $\psi$. 
Then the result follows from Lemma \ref{lem-bino-poly-mom}.
\end{proof}

\begin{theorem}\label{thm2_ass5}
Suppose the tau leap method generates $K_j$ for $j=1,\dots,M$ to be
independent conditioned on current state $x$ and each $K_j$ is either
binomially or Poisson distributed with their distributions satisfying the
assumptions of Lemmas \ref{lem-bino-deriv-bnd} and \ref{lem-poiss-deriv-bnd}. 
Then the hypotheses of Theorem \ref{thm_Ass5} are satisfied and thus
Assumption 5 holds.
\end{theorem}
\begin{proof}
By the assumed independence of $K_j$ it follows that $\tilde{\phi}$ has a
product form
\[
\tilde{\phi}(\tau,x;k) = \tilde{\phi}_1(\tau,x;k_1) \dots
\tilde{\phi}_M(\tau,x;k_M).
\]
Then for $i=0,1,\dots,q+1$ the $i$th derivative $\tilde{\phi}^{(i)}(\tau,x;k)$ is a linear combination
of terms of the form
\[
\tilde{\phi}^{(i_1)}_1(\tau,x;k_1) \dots \tilde{\phi}^{(i_M)}_M(\tau,x;k_M),
\]
where $i_j \in \{0,1,\dots,q+1\}$ for $j=1,\dots,M$. Noting that 
\[
\begin{aligned}
\sum_k |k|^r |\tilde{\phi}^{(i)}| &= \sum_k (k_1 + \dots +k_M)^r |\tilde{\phi}^{(i)}|\\
&\leq M^r \sum_{k_1} \sum_{k_2} \dots \sum_{k_M} (k_1^r + \dots + k_M^r)
|\tilde{\phi}^{(i_1)}_1| \dots |\tilde{\phi}^{(i_M)}_M|
\end{aligned}  
\]
the result follows from using Lemmas \ref{lem-bino-deriv-bnd} and \ref{lem-poiss-deriv-bnd}.
\end{proof}
 
\begin{corollary}\label{cor_ass3}
Suppose that the conditions of Theorem \ref{thm2_ass5} and the extra
conditions of \ref{thm_Ass3} hold. Then Assumption 3 holds.
\end{corollary}
\begin{proof}
The conditions of Theorem \ref{thm_Ass5} are implied by conditions of Theorem
\ref{thm2_ass5}. Given the extra conditions of Theorem \ref{thm_Ass3} the
conclusions of Theorem \ref{thm_Ass3} follow. 
\end{proof}

\begin{lemma}
\label{lem-poiss-deriv}
Let $K$ be Poisson distributed with parameter $\lambda$ where $\lambda=\lambda(x,\tau)$
is a function of state $x \in \posint^N$ and step size $\tau \geq 0$. 
Suppose that there exists $\delta>0$ such that for all $x \in \posint^N$ and $\tau \in [0,\delta]$, $\lambda$ is continuously differentiable
in $\tau$, and the supremum of $\lambda,|\lambda^{(1)}|$ over
$\tau \in [0,\delta]$ is bounded above by a polynomial of degree $s$ in $|x|$.
Then for $\tau \in [0,\delta]$ and for each $r \in \nat$ 
the supremum of $|d E(K^r)/d \tau|$ over $\tau \in [0,\delta]$ 
is bounded by a polynomial of degree $rs$ in $|x|$. 
\end{lemma}  

\begin{proof}
For a fixed $x \in \posint^N$, the random variable $K(x,\tau)$ is a time 
non-homogeneous Poisson process in $\tau$ with rate (intensity) 
$\lambda^{(1)}(x,\tau)$. It follows that 
\[
d E(K^r)/d \tau = \lambda^{(1)} E\{(K+1)^r -K^r\}.
\]
This together with Lemma \ref{lem-poiss-poly-mom} implies the desired result. 
\end{proof}

\begin{lemma}
\label{lem-bino-deriv}
Let $K$ be binomially distributed with parameters $N_0$ and $p$ where $N_0=N_0(x)$
is a function of state $x \in \posint^N$ and $p=p(x,\tau)$ is a function of
state $x$ and step size $\tau \geq 0$.
Suppose that there exists $\delta>0$ such that for all $x \in
\posint^N$ and $\tau \in [0,\delta]$, $p$ is continuously differentiable
in $\tau$, and the suprema of $|p^{(1)}|$ over
$\tau \in [0,\delta]$ and $N_0(x)$ are bounded above by polynomials of degree 
$s_1$ and $s_2$ respectively in $|x|$.
Then for $\tau \in [0,\delta]$ and for each $r \in \nat$ 
the supremum of $|d E(K^r)/d \tau|$ over $\tau \in [0,\delta]$ 
is bounded by a polynomial of degree $s_1 + r s_2$ in $|x|$. 
\end{lemma}  

\begin{proof}
We write $K=K_{N_0}$. Using the relationship mentioned in the proof of Lemma \ref{lem-bino-deriv-bnd}
we obtain that
\[
d E(K_{N_0}^r)/d \tau = N_0 p^{(1)} E\left((K_{N_0-1}+1)^r-K_{N_0-1}^r\right). 
\]
This together with Lemma \ref{lem-bino-poly-mom} implies the result. We note that since $p$ lies in $[0,1]$ we only need to focus on dependence on $N_0$ and $p^{(1)}$.
\end{proof}

\begin{theorem}\label{thm2_ass6}
Suppose that there exists $\alpha>0$ satisfying the hypothesis of Theorem \ref{thm_ass2} and 
that the $K_j$ for $j=M_s+1,\dots,M$ corresponding to the linearly bounded reactions are
(conditioned on current state $x$) are each either
binomially or Poisson distributed with their distributions satisfying the
assumptions of Lemmas \ref{lem-bino-deriv} with $s_1=0$ and $s_2=1$ or \ref{lem-poiss-deriv} with $s=1$ respectively.
 
Additionally suppose that for $x \in \posint^N$ and for $j=1,\dots,M$ that 
$x+\nu K^{(1)} \in \posint^N$ with probability $1$ where $K^{(1)}$ is the $M_s$ vector of the superlinear reaction counts per tau leap. Also suppose that for $x \notin \posint^N$ we have that $K_j=0$ with probability $1$ for all $j$.  
Then the hypotheses of Theorem \ref{thm_ass6} are satisfied and thus
Assumption 6 holds.
\end{theorem}
\begin{proof}
These assumptions guarantee that with $K^{(2)}=(K_{M_s+1},\dots,K_M)$, 
\[
|d E( |K^{(2)}|^r) /d \tau| \leq \beta_r (1+|x|^r),
\]
for some $\beta_r$ independent of $\tau \in [0,\delta]$ and $x$. Since for $\tau=0$ we have $E(|K|^r)=0$, using mean value theorem we obtain the bounds
\[
E(|K^{(2)}|^r)=m_r \leq \beta_r (1+|x|^r) \tau, \quad \tau \in [0,\delta].
\]
Thus all the assumptions of Theorem \ref{thm_ass6} are satisfied and thus Assumption 6 holds.
\end{proof}

\subsection{Example}
We consider the example of the unbounded reaction system
\begin{equation}
S_1 + S_2 \to S_3, \;\; S_3 \to S_1 + S_2 \;\; S_2 \to 2 S_2, \;\; S_2 \to 0,
\label{ex1}
\end{equation}
where the propensities are assumed to be of the stochastic mass action form:
\begin{equation}
a_1(x) = c_1 x_1 x_2, \; \; a_2(x) = c_2 x_3, \;\; a_3(x) = c_3 x_1, \;\; a_4(x) = c_4 x_2.
\end{equation}
We note that Assumption 1 is clearly satisfied. 

The stoichiometric vectors are $\nu_1 = (-1,-1,1)^T$, $\nu_2 = (0,1,-1)^T$, 
$\nu_3 = (0,1,0)^T$ and $\nu_4=(0,-1,0)^T$. 
%Taking $\alpha = (1,0,1)^T$ we see 
%that $\alpha^T \nu_j \leq 0$ for all $j$. 
%Thus by Theorem \ref{thm_stoich_bdd} we note that 
It is easy to see that $S_1$ and $S_3$ are bounded (if initial conditions are bounded) 
as $(1,0,1)^T \nu_j \leq 0$ for all $j$ implying that $X_1(t) + X_3(t) \leq X_1(0) + X_3(0)$.  
However $S_2$ is not bounded because of reaction $3$ and thus the system is unbounded.  
However since $\alpha=(1,1,1)^T$ satisfies the hypothesis of Theorem \ref{thm_ass2} we
see that Assumption 2 is satisfied. 

Suppose we use a tau leap update following the REMM-$\tau$ method \cite{rathinam-elsamad-REMM-JCP}:
\[
X(t+\tau) = X(t) + \sum_{j=1}^M \nu_j K_j
\]
where $K_1 \sim \text{Binomial}(N_1,p_1)$, $K_2 \sim \text{Binomial}(N_2,p_2)$,
$K_3 \sim \text{Poisson}(\lambda_3)$ and $K_4 \sim \text{Binomial}(N_4,p_4)$,
where $K_j$ are all independent conditioned on $X(t)=x \in \posint^N$ and
\[
\begin{aligned}
N_1 &= \min\{x_1,x_2\}, \;\; p_1 =  \frac{\tilde{c}_1}{\tilde{c}_1+c_2}
(1-e^{-(\tilde{c}_1+c_2})\tau),\\
N_2 &= x_3, \;\; p_2 = \frac{c_2}{\tilde{c}_1+c_2}
(1-e^{-(\tilde{c}_1+c_2})\tau),\\
\lambda_3 &= \frac{c_3 x_2}{c_4} (1 - e^{-c_4 \tau}),\\
N_4 &= x_2, \;\; p_4 = (1-e^{-c_4 \tau}),
\end{aligned}
\]
where 
\[
\begin{aligned}
\tilde{c}_1 &= (\max\{x_1,x_2\}+1)c_1 \, \text{if } \min\{x_1,x_2\}=0,\\
\tilde{c}_1 &= \max\{x_1,x_2\} c_1 \, \text{else}.\\
\end{aligned}
\]
If $X(t)=x \notin \posint^N$ then we set $K_j=0$ for all $j$ and the update is
$X(t+\tau)=x$. We note that this particular step differs from the way negativity was handled in \cite{rathinam-elsamad-REMM-JCP}, but freezing the tau leap process once it leaves $\posint^N$ allows for easier verification of Assumption 6 
as stated in Theorem \ref{thm_ass6}.

It is clear that $N_1,N_2$ and $N_4$ are bounded by a polynomial in $|x|$. 
It is also clear that $p_1,p_2,p_4$ and $\lambda_3$ are infinitely
differentiable and the maximum of their derivatives on any bounded interval
$[0,\delta]$ of $\tau$ is also bounded by a polynomial in $|x|$. Thus the
hypotheses of Theorem \ref{thm2_ass5} are satisfied and hence Assumption 5
holds. 

The REMM-$\tau$ method was designed to satisfy the conditions that
\[
\begin{aligned}
\tilde{\phi}^{(1)}(0,x;0) &= -a_0(x),\\
\tilde{\phi}^{(1)}(0,x;e_j) &= a_j(x), \;\; j=1,\dots,M,\\
\tilde{\phi}^{(1)}(0,x;k) &= 0, \;\; k \notin \{0,e_1,e_2,\cdots,e_M\},
\end{aligned}
\] 
which can be directly verified by differentiation the details of which we shall omit. 
Thus by Corollary \ref{cor_ass3} pointwise consistency Assumption 3 follows. 

As Assumptions 1 through 5 hold, by Theorem \ref{thm_totv_conv} the method is 
first order ($O(\tau)$) convergent in total variation.  

In order to verify Assumption 6 we shall verify the conditions of Theorem \ref{thm2_ass6}. Firstly we note that the only superlinear reaction is $1$, and that 
as $0 \leq K_1 \leq \min\{x_1,x_2\}$ it is clear that starting from a state $x \in \posint^N$ the state reached after the update $x + \nu_1 K_1$ still remains in $\posint^N$.  

We note that $|\lambda^{(1)}| \leq c_3 x_2$ and hence can take $s=1$ in Lemma \ref{lem-poiss-deriv} regarding $K_3$. Also we note that $N_2(x) \leq |x|$ and $N_4(x) \leq |x|$
and $|p_2^{(1)}| \leq c_2$ and $|p_4^{(1)}| \leq c_4$. Thus we can take $s_1=0$ and $s_2=1$ regarding both $K_2$ and $K_4$ in Lemma \ref{lem-bino-deriv}. 
Hence all the conditions of Theorem \ref{thm2_ass6} are satisfied     
and we can conclude that Assumption 6 holds and hence by Theorem \ref{thm_conv} the method is first order convergent in $r$th moment variation for each $r \in \nat$. 
This also implies the convergence of all moments.   

\section{Discussion of results and concluding remarks}
For the purpose of this discussion we need to differentiate the type of 
convergence considered in this paper from the type of analysis which 
relates $\tau$ to system size $V$ as $\tau=V^{-\beta}$ and studies convergence 
as $V \to \infty$. We shall refer to the former as 
{\em convergence in fixed system sense} and the latter as 
{\em convergence in large system limit}.  

While our (fixed system sense) convergence results were stated for general
order of convergence $O(\tau^q)$, we have not seen a practical tau leap method
that is $O(\tau^2)$ convergent in general in the fixed system sense. 
The {\em weak trapezoidal method} mentioned in \cite{Hu-Li+JCP11, Anderson-Koyama-MMS12}
was shown to be 2nd order consistent under the restrictive assumption that 
$\xi_1 a_j(x+\nu_k)- \xi_1 a_j(x) \geq a_j(x)$ for all
$x,j$ and $k$ where $\xi_1 \in [2,\infty)$ is a method parameter. 
This leads to the condition that $a_j(x+\nu_k) \geq a_j(x) (1 - 1/\xi_1)$ for
all $x,j,k$. When $x$ is on the boundary of $\posint^N$ this may not hold for
most systems. However, if with probability close to $1$ the system state is
far away from the ``bad'' boundaries, then one expects this method to be more
accurate and for this to be valid one expects the system size to be large. The {\em midpoint tau} method is shown to be $O(\tau^2)$ convergent when $V \to
\infty$ with $\tau = V^{-\beta}$ \cite{Anderson-Ganguly+AAP11}. However, midpoint method is
only first order convergent in the fixed system sense. In practice, for 
modestly large molecular copy numbers one may expect 
higher accuracy for both these methods(than the explicit tau leap), 
while for low copy numbers one may still expect these methods to be well 
behaved because they are first order convergent in the fixed system sense. 

As a general rule, if a tau leap method shows higher order accuracy in
the large system limit and is first order convergent in the fixed system sense 
it will be expected to be more effective than the first order convergent 
explicit tau. On the other hand if a method is higher order convergent in
the large system limit, but is non-convergent or (even worse) not zero stable 
in the fixed system sense then the method should not be used. 

It is easy to come up with higher order accurate (in the fixed system sense)
tau methods that may not be practical. For instance one may take the tau
update probabilities $\tilde{\phi}(\tau,x;k) = \text{Prob}(K=k \, | \, Y(t)=x)$ 
to agree with exact probabilities $\tilde{p}(\tau,x;k)$ up to $O(\tau^q)$ for 
the case of $|k|=1,\dots,q$, set  $\tilde{\phi}(\tau,x;k) = 0$ for $|k| \geq q+1$ 
and set $\tilde{\phi}(\tau,x;0)$ accordingly. (We note that
$\tilde{p}(\tau,x;k)=O(\tau^{|k|})$, see \cite{rathinam+05MMS} for instance).
Such a naive approach will result in a $O(\tau^2)$ convergent method that 
will leap over at most two reaction events ($q$ events for the case of order $q$), not to mention other practical issues 
that need to be dealt with such as truncated Taylor expansions being for probabilities being non-negative. 

The analysis in this paper does not suggest new tau leap methods. However it 
does provide some guidance to ensure that a tau leap method is convergent and 
zero stable (in the fixed system sense) so that the user does not have to worry about the small step sizes resulting in large errors. The most delicate of the 
assumptions is Assumption 6 which implies zero stability of the tau leap
method (in term of moments). Zero-stability may not be taken for granted. We refer to \cite{Hutzenthaler+10} for an example
(in the case of SDEs driven by Brownian motion) showing lack of convergence 
(and lack of zero stability) of the moments of the Euler method. Theorem \ref{thm_ass6} provides sufficient conditions under which 
Assumption 6 can be verified and suggests that it is best to use bounded
random variables (such as Binomials) in the tau update of superlinear reactions (see Remark \ref{rem-ass6}).   

Finally we like to note that finding a tau leap method 
that is $O(\tau^2)$ convergent {\em uniformly in system size} $V$ (after a
suitable scaling by a power of $V$) might prove to be useful. 
The error estimates derived in 
\cite{Anderson-Ganguly+AAP11, Hu-Li+JCP11, Anderson-Koyama-MMS12} contain
system size $V$ and step size $\tau$ (under the bounded system condition 
and/or global Lipschitz condition on propensities). 
None of the methods presented there are $O(\tau^2)$ convergent uniformly 
in $V$. We believe that the analysis in this paper can be extended to 
include the dependence of the error in the moments of a tau leap method on $V$
and $\tau$ for the case of unbounded systems with nonlinear but polynomial 
growth propensities. While this exercise will not automatically result in a  
``$O(\tau^2)$ convergent uniformly in $V$'' method, it will help provide some
insights towards the construction of such methods.  

%{\bf Acknowledgment:}
%We like to thank the anonymous referees for the constructive feedback that helped improve this manuscript. 
  
\bibliography{references}
\bibliographystyle{siam}

\end{document}